\documentclass[12pt]{amsart}             
\usepackage[fontsize=14]{scrextend}        

\newtheorem{theorem}{Theorem}[section]

\newtheorem{corollary}[theorem]{Corollary} 
\theoremstyle{definition}
\newtheorem{definition}[theorem]{Definition}

\usepackage{fullpage} 
\usepackage[all]{xy} 
\usepackage{enumerate,dsfont,txfonts,mathptmx,txfonts,newtxtext}   

\usepackage[utf8]{inputenc}
\usepackage[T1]{fontenc}
\usepackage{pdfrender}

\usepackage[dvipsnames]{xcolor}   
\usepackage{xr-hyper}
\usepackage[linktocpage=true,colorlinks=true,hyperindex,citecolor= Royal Blue,linkcolor=Royal Blue]{hyperref}

\begin{document}
\pdfrender{StrokeColor=black,TextRenderingMode=2,LineWidth=.01pt}

\title{Multidimensional population modeling with  transition structures}

\author{Amartya Goswami}

\address[a]{Department of Mathematics and Applied Mathematics,
University of Johannesburg, 
P.O. Box 524, 
Auckland Park 
2006, 
South Africa}

\address[b]{National Institute for Theoretical and Computational Sciences (NITheCS), South Africa.}

\email{agoswami@uj.ac.za}

\begin{abstract}
The aim of this paper is to describe a population model with transition. We analyze the spectral properties of the transition matrix considering both irreducible and reducible structures. We give physical interpretations of these properties to population dynamics.
\end{abstract}

\makeatletter
\@namedef{subjclassname@2020}{%
\textup{2020} Mathematics Subject Classification}
\makeatother


 
\maketitle 

\section{Introduction}
Formal demography is concerned with the mathematical description of human (or nonhuman) populations, particularly with the demographic processes such as birth, death, aging \textit{etc}., and its long time  behaviour. Our work is based on models, which deals with the extension of that focus to include the transition of the members of the population among several regions. One of the first major contribution along this line is by Rogers \cite{Ro1}.

Following \cite{Ro2} and \cite{Ro3}, ``formal multi-regional demography, therefore, is concerned with the mathematical description of the evolution of populations over time and space. The trifold focus of such descriptions is on the \textit{stocks} of population groups at different points in time and locations in space, the vital \textit{events} that occur among these populations, and the \textit{flows} of members of such populations across the spatial borders that delineate the constituent regions of the multi-regional population system.
A biological population may experience multiple states in two ways:

First, it may visit different states in the course of time, the whole population experiencing the same (possibly age-specific) vital rates at any one time. For example, a troop of baboons moves from one area to another of its range, with associated changes in food supply and risks of predation \cite{Al}. A human population experiences fluctuating crop yields from one year to the next, with associated effects on childbearing and survival. These are \textit{serial} changes of state of a \textit{homogeneous} population.

Second, the population may be subdivided into \textit{inhomogeneous} subpopulation that experience different states in \textit{parallel}. Individuals may migrate from one state to another in the course of time. The states may corresponds to geographical regions, work status, marital status, health status, or other classifications \cite{Ro3}. Individuals within a given state at a given time are assumed to be homogeneous with respect to their vital rates.  In general multistate stable population theory is used for two reasons:
\begin{enumerate}
[\upshape (i)]
\itemsep -.2em\itemsep -.2em
\item To explore the relation between individual life histories and  population characteristics, and
  
\item to explore the consequences of current life histories and variations in life histories. Stable populations are a population that experiences particular demographic regimes (transitions) over a long period. The theory is used to magnify the effects of a current demographic regime and to assess the consequences of small changes in demographic behaviour (microscopic view).''
\end{enumerate}
For more detailed study of multidimensional demography, we refer to \cite{Ro2}, \cite{LR}.

We arrange our work as follows: First we describe the population model with all the necessary assumptions. Next we  analyze the spectral properties of our transition matrix using Perron-Frobenious theorem. Here we will do the spectral analysis separately for two cases, namely irreducible and reducible. We give the physical significances of these spectral properties. At the conclusion, we point out some of the possible generalisations of this work. 
  
\subsection{The model}
We are going to  consider a closed (with no immigration and emigration), one-sex population model which is linear and continuously depends on  age. The whole population is additionally subdivided into several groups. We assume the population is distributed into $n$ number of patches or states which may be interpreted as geographical areas or any other classifications with proper meaning. The basic assumptions
regarding this population model are as follows:

\begin{enumerate}[\upshape (i)]
\itemsep -.2em\itemsep -.2em
\item Migrations between patches occur in a discrete manner.

\item The migrations occur on a much faster time scale than the
demographic events
(ageing, births and deaths).

\item The migration rates between the states, death rates and
the birth rates may be age dependent.

\item The movement process between the patches is conservative
with respect to the life dynamics of the population.
\end{enumerate}

To describe the model, let $u_i (a, t)$, for $1 \leqslant i \leqslant n $,  be the population at time $t$ of individuals residing in patch $i$ and being of age $a$ so that the number of inhabitants in the age interval of $a + da$ is approximately equal to $u_i(a, t) da$ and precisely $ \int_{a}^{a + da} u_i(s, t) \,ds.$ Let $M(a):= \text{diag}[\mu_i(a)]_{1 \leqslant i \leqslant n}$ denote the mortality matrix in which $\mu_i(a)$ is the mortality rate at age $a$ in the $i$th patch. Let the (transition) matrix $C:=[c_{ij}]_{1 \leqslant i,j \leqslant n}$ describe the transfer of individuals between patches, that is,  $c_{i j}$, for \,$i \neq j $, denote the transition rates  from state $j$ to state $i$ and the diagonal elements $c_{i i}$,  defined as $c_{i i}:=- \sum_{i =1}^{n} c_{j i}$  indicate the net loss of population in $i$th state due to transition to other states. The matrix $C$ is known as \emph{Kolmogorov matrix}. Let $B(a):=[\beta_i(a)]_{1 \leqslant i \leqslant n}$ denote the fertility matrix  in which $\beta_i(a)$ is the average number of offspring of state $i$ per unit time produced by an individual at age $a$. Let $\phi(a)$ be the initial population vector. Migration between the patches occur at a much faster time scale than the demographic processes such as aging, and this is incorporated by introducing a large parameter $1/\epsilon$ multiplying  the transition matrix $C$. For all our future work, we will assume that $\epsilon > 0$ is an arbitrary small number.

Denoting the population vector as $u(a,t):= (u_1(a,t),\dots,u_n(a,t))$  we have the following system of equations (in vector form) with proper initial and boundary conditions
\begin{align*}
u_t =& - u_a - Mu+ \frac{1}{\epsilon} \,Cu,
\\
u(0, t) =&   \int_{0}^{\infty} B(s)u(s,t)\,ds,
\\
u(a,0) =&  \phi(a),
\end{align*}
where subscripts $t$ and $a$, respectively denote differentiation with respect to the variables $t$ and $a$.  From the practical point of view, one can always assume a finite upper limit of fertility age (say, $w$)  and the  rest of the range of integral (\textit{i.e.}, from $w$ to infinity) as zero. This model has been studied in \cite{Ar1}, \cite{Ar}, \cite{Br}, \cite{Li}, where this model has been approximated to the averaged model. In \cite{BGS}, a detailed error analysis of this approximation has been done. All of these works based on the assumption of irreducibility of the transition matrix.
 
\section{The transition matrix}
Let us first assume that there is a movement of the members between any two patches of the region and it is observed that in long run, the population stabilizes.This type of behaviour can be represented with an irreducible matrix. But from practical point of view it is also quite reasonable that a group from the population dispersed from one region to other regions. For example, due to economic condition or natural disasters, people move from one part of a country to other parts and the left part remains without any people for a long time. Also with age, old people hardly moves to other regions and they stay remain at the same place but their off-springs moves else where. Therefore considering all these issues, we can think of a transition matrix in which some patches are confined in the sense that once a group enter these patches, never leaves those patches. This type of transition phenomenon can be modeled with reducible matrix.  

In the next two sections, we are going to discuss the detailed spectral properties of our transition matrix as the long time behaviour of a population depends on the dominant eigenvalue and the corresponding eigenvector of the transition matrix. We will  look on irreducible and reducible cases separately. But, first let us give some necessary definitions.
 
\begin{definition} 
An $m \times n$ matrix $A = [a_{ik}]$ with real elements is called \emph{non-negative} $(A \geqslant 0)$ or \emph{positive} $(A >0)$, if all the elements of the matrix $A$ are non-negative $($respectively positive$)$: $a_{ik} \geqslant 0$ $($respectively $>0$$)$. Similarly a vector $k:=(k_1, \ldots, k_n)$ is said to be non-negative $(k \geqslant 0)$ or positive $(k>0)$ accordingly $k_i \geqslant 0$ or $k_i >0$, where $1 \leqslant i \leqslant n$.
\end{definition}
 
\begin{definition}
Let $A$ be an $n \times n $ matrix. The set of eigenvalues of A, called \emph{spectrum} of $A$, will be denoted by $\sigma(A)$. The \emph{spectral radius} of $A$ is defined as $\text{Sp}(A):= \max\{|\lambda| : \lambda \in \sigma(A)  \} $. The \emph{spectral bound} of $A$ is defined as $s(A):= \max\{\mathfrak{R} \lambda : \lambda \in \sigma(A) \}$, where $\mathfrak{R} \lambda $ denotes the real part of $\lambda$.
\end{definition}
 
\begin{definition}
Let $A$ be an $n \times n$ matrix. An eigenvalue $\lambda$ of $A$ is called a \emph{simple eigenvalue} of $A$ if its algebraic multiplicity is one. An eigenvalue $\lambda_0$ of $A$ is called a \emph{dominant eigenvalue} of $A$ if $\lambda_0 > \mathfrak{R} \lambda$ for all $\lambda \in \sigma(A)\setminus \{\lambda_0\}$ and the corresponding eigenvector is called a \emph{Perron vector}.
\end{definition}

\begin{definition}
The matrix $A = [a_{ik}]_{1\, \leqslant \,i,\,k \, \leqslant \, n}$ is called \emph{reducible}, if there is a permutation of the indices which reduces it to the form $$\tilde{A}:= \left[\begin{array}{lr} B & D \\ 0 & T\end{array}\right], $$ where $B$ and $T$ are square matrices. Otherwise the matrix $A$ is called \emph{irreducible}.
\end{definition} 

\subsection{Irreducible transition matrix}
We shall first consider irreducible matrix structure. An irreducible transition matrix does have special spectral properties which we describe in Theorem \ref{402}.  Results \smash{(Theorem \ref{PF})} are available only for irreducible non-negative matrices. Since our transition matrix has non-negative off-diagonal elements, we need specific results for this type of matrices which can be obtained using Theorem \ref{PF} and we prove them in \smash{Corollary \ref{abc}}.

\begin{theorem}
{\rm{(Perron--Frobenius)}} An irreducible non-negative matrix always has a positive dominant, simple eigenvalue. Corresponding to this dominant eigenvalue, there is an eigenvector with positive coordinates.
\label{PF}
\end{theorem}

\begin{proof} 
See \cite[p. 53]{G}. 
\end{proof}

\begin{corollary}
{\rm{\cite{S}}} A matrix $A$ with non-negative off-diagonal entries has the following properties:
\begin{enumerate}[\upshape (i)]
\itemsep -.2em
\item The spectral bound $s(A)$  is in $\sigma(A)$ and there is a vector $v > 0$ $($\textit{i.e.} each component of $v$ is positive$)$ such that \smash{$Av = s(A)v$}.

\item $\mathfrak{R} \lambda <s(A)$ for all $\lambda \in \sigma(A)\setminus \{s(A)\}$.
 
\item If, in addition, $A$ is irreducible, then $s(A)$ has algebraic multiplicity one.
\end{enumerate} 
\label{abc}
\end{corollary}

\begin{proof} $A + c I \geqslant 0$ for all large $c$  \,since $A$ has non-negative off-diagonal entries. Therefore Theorem \ref{PF} applies to $A + c I$ for such $c$. In particular, $\text{Sp}(A + c I)$ is positive and it is an eigenvalue of $A + c I$ for all large $c$. Moreover there exists a corresponding non-negative eigenvector. Since adding $c I$ to $A$ results in $\sigma (A + c I ) = c + \sigma(A)$, it follows from Theorem \ref{PF}, that $s (A + c I) = \text{Sp}(A + c I) = s(A) + c$ and therefore $s (A)$ is an eigenvalue of $A$. Also, adding $c I$ to $A$ preserves the property of irreducibility in the sense that $A$ is irreducible if and only if $A +c I$ is irreducible. 
\end{proof}

Using Corollary \ref{abc}, we can now prove the following spectral properties of a Kolmogorov matrix $C$.

\begin{theorem}
An irreducible transition matrix $C$ has zero as a simple, dominant eigenvalue and the rest of the eigenvalues have negative real parts. Moreover, $1$ is a left eigenvector corresponding to the zero eigenvalue of $C$. 
\label{402}
\end{theorem}

\begin{proof} 
The fact that $C$ has zero as an eigenvalue follows from the fact that the sum of the entries of its column is zero.  From Corollary \ref{abc}, it follows that zero is a simple eigenvalue of $C$.  Let $k$ be the corresponding right eigenvector of the zero eigenvalue. Now let us prove the rest of the part; that is, zero is a dominant eigenvalue. Following Corollary \ref{abc}, we have $$Ck = s(C)k,$$ which, in fact is a system of linear algebraic equations. Summing over this system of equations we obtain (due to the property of the matrix $C$) $$0 = s(C) ( k_1 + \cdots + k_n ),$$ where $k :=(k_1, \ldots, k_n )$.
The reason of getting zero on the left hand side is that the sum of the entries of each column  of $C$ is zero. Since  from \smash{Corollary \ref{abc}}, $k>0$, we  must have $s(C) = 0$.
Since $s(C)$ is defined as the largest number among the real parts of the eigenvalues of $C$, the result is proved.
Since the sum of the elements of a column of $C$ is zero, it follows that $1$ is a left eigenvector corresponding to the zero eigenvalue of $C$ and this completes the proof. \end{proof}

\subsection{Reducible transition matrix}
The situation changes quite a lot for a reducible matrix. Instead of \smash{Theorem \ref{PF}}, we have the following result, where the dominant eigenvalue is not any more simple and the corresponding eigenvector looses its positivity.

\begin{theorem}
The spectral radius of a non-negative matrix $A$ is a dominant eigenvalue   of $A$ and there is a corresponding eigenvector $k\geqslant 0$.
\label{P-F}
\end{theorem}

\begin{proof} We refer to \cite[p. 66] {G}. \end{proof}

\begin{corollary}
The spectral bound of a non-negative off-diagonal matrix $A$ is an eigenvalue of $A$ and there is a corresponding eigenvector $k \geqslant 0$.
\label{ab}
\end{corollary}


Using Corollary \ref{ab}, we can prove the following spectral properties of a reducible transition matrix. Since the proof is essentially same as that of Theorem \ref{402}, we are not repeating it here again.

\begin{theorem}
A reducible transition matrix has zero as the dominant eigenvalue and $1$ is a left eigenvector corresponding to this dominant eigenvalue.
\label{Ma}
\end{theorem}

Due to the reducibility property, the left and right Perron vectors have some special structures as described in the following theorems.

\begin{theorem}
An arbitrary reducible, non-negative matrix cannot have both left and right positive eigenvectors corresponding to the dominant eigenvalue unless the matrix is block diagonal.
\label{Ga}
\end{theorem}

\begin{proof} See \cite[p. 78] {G}.  \end{proof}

\begin{theorem}
For a reducible transition matrix,  the right Perron vector must contain some zero elements unless the matrix is  block diagonal.
\label{500}
\end{theorem}

\begin{proof} The proof immediately follows from Theorem \ref{Ma} and as $1$ is a left eigenvector corresponding to the zero eigenvalue. 
\end{proof}


Considering  reducible transition matrix, next we will see the possible connection between the structure of the matrix and the right Perron vector. Using the definition, a reducible matrix $C$ can be represented as
\begin{equation}
\left[\begin{array}{lr}
C' &  A \\ 0 & T
\end{array}\right],
\label{nor}
\end{equation}
where $C'$, $T$ are irreducible, square matrices. Also since $C$ is a transition matrix and as we have $0$ below $C'$, we find that $C'$ has zero as a simple dominant eigenvalue. The spectrum of $T$ has the property that $\sigma(T):= \{ \lambda \in \mathds{C} : \mathfrak{R}\, \lambda < 0 \}$.
If one of the matrices $C'$ or $T$ is reducible then it can also be represented in a form similar to (\ref{nor}), and we can carry on the above decomposition. Finally with suitable permutations,  $C$ can be represented in the following normal form:
\begin{equation}
\left[\begin{array}{cccccc}
C_{n_1} &  \cdots & 0 & A_{n_1,\,n_{m+1}} & \cdots & A_{n_1,\,n_n}  \\ \vdots & \vdots & \vdots & \vdots & \vdots & \vdots \\ 0 &  & C_{n_m} & A_{n_m,\,n_{m+1}} & \cdots & A_{n_m,\,n_n} \\ 0 & \cdots & 0 & T_{n_{m+1},\,n_{m+1}} & \cdots & T_{n_{m+1},\,n_n} \\ \vdots & \vdots & \vdots & \vdots & \vdots & \vdots \\ 0 & \cdots & 0 & 0 & \cdots & T_{n_n,\,n_n}
\end{array}\right],
\label{Nor}
\end{equation}
where  $C_{n_i}$, for $1 \leqslant i \leqslant m$,  are irreducible matrices and in each row $$A_{n_i, n_{m+1}}, \ldots, A_{n_i, n_n} \qquad  (1 \leqslant i \leqslant n),$$ at least one matrix is different from zero.
There is an interesting connection between $C_{n_i}$ and $T_{j,j}$  with the right Perron vector of a reducible transition matrix as described in the following result.

\begin{theorem}
For a reducible transition matrix in its normal form, $C_{n_i}$  correspond to the non-zero elements of the right Perron vector and $T_{j,j}$ correspond to the zero elements of the right Perron vector.
\label{800}
\end{theorem}

\begin{proof} From the property of the normal form, all  $C_{n_i}$, where $1 \leqslant i \leqslant m $, have zero as the dominant eigenvalue and all  $T_{j\,,j}$, where \smash{$n_{m+1} \leqslant j \leqslant n_n $,} have negative real parts. Let all  square matrices $C_{n_i}$ have orders, respectively, $n_i$, where $1 \leqslant i \leqslant m $ and all  $T_{j, j}$ have orders, respectively $n_j$, where $n_{m+1} \leqslant j \leqslant n_n $.  Now $C$ has $1$ as a left Perron vector and, by Theorem \ref{500}, a right Perron vector of $C$ must have some zero components.  Let $k:= (k_{1}, k_{2}, \ldots, k_n)$ be a right Perron vector of $C$, where $k_1:= (k_{11}, \ldots, k_{n_1})$, $\ldots$, $k_m:= (k_{n_{m-1} + 1}, \ldots, k_{n_m})$,
$k_{m+1}:= (k_{{n_m}+1}, \ldots, k_{n_{m+1}})$, \ldots,  $k_n:= (k_{n_{n-1}+1}, \ldots, k_{n_n})$.
Corresponding to the zero eigenvalue, we  have the following set of equations:
\begin{eqnarray}
C_{n_1} k_{1} + \cdots + 0 + A_{n_1,\, n_{m+1}} k_{m + 1} + \cdots + A_{n_1,\, n_n} k_n &=&  \nonumber 0, \\\nonumber \vdots\\\nonumber 0 + \cdots + C_{n_m} k_m + A_{n_m,\, n_{m+1}} k_{m + 1} + \cdots + A_{n_m,\, n_n} k_n &=&  0, \\\nonumber 0 + \cdots + 0 + T_{n_{m+1},\, n_{m+1}} k_{m+1} + \cdots + T_{n_{m+1},\, n_n}k_n &=& 0, \\\nonumber\vdots\\\nonumber 0 + \cdots + 0 + 0 + \cdots + T_{n_n,\,n_n} k_n &=& 0.
\end{eqnarray}
Since  each $T_{j,j} \, ( n_{m+1} \leqslant j \leqslant n_n)$ is non-singular, it is invertible and hence $k_{m+1} = \cdots = k_n = 0$. Since each  $C_{n_i}$ is irreducible,  it has positive Perron vector corresponding to the dominant eigenvalue zero and hence each  $k_1, \ldots, k_m$ is non-zero. This completes the proof. 
\end{proof}

\section{Interpretations to population dynamics}
The above spectral analysis can be used to model transition phenomenon in population dynamics. Existence of the dominant eigenvalue plays a crucial role in long time behaviour and positivity of the corresponding eigenvector (in irreducible case) gives the stable distribution where as the situation is quite different with a reducible structure.

\subsection{Irreducible transition matrix}
Let us first assume that our transition matrix $C$ is irreducible. Biological heuristics suggests that no geographical structure should persist for very large interstate transition rates; that is, for $\epsilon \to 0$. Here we also note that both biological and mathematical analysis rely on $\lambda = 0$ being the dominant, simple eigenvalue of $C$ for each with the corresponding positive eigenvector, denoted by $k$, and the left eigenvector $1$. Vector $k$ is normalized to satisfy $1 \cdot k = 1$ and $k = \left( k_1, \ldots, k_n \right)$ is the so-called stable patch structure; that is, the asymptotic (as $t \to \infty$ and disregarding demographic processes) distribution of the population among the patches.  Thus, in population theory, the components of $k$ are approximated as  $k_i \approx u_i/u$ for $i = 1, \ldots, n$, where
$u := \sum_{i=1}^n u_i.$ For a detailed asymptotic analysis of this model, we refer to \cite{BGS}.

\subsection{Reducible transition matrix}
With respect to transition states,  reducibility of the transition matrix $C$ produces different behaviour of what compared to previous case.  For our transition matrix $C$, we divide the states into two sets:
the states $\mathfrak{T}$ from which systems can make transit to any other states are called \textit{transient states} and the set of states $\mathfrak{C}$  are called \textit{closed states}, where once the system is in, cannot transit to other states no mater how long we iterate. Therefore, from the normal form \ref{Nor}, we find that all the $C_{n_i}$, where $1 \leqslant i \leqslant m$, corresponds to the closed states and all the $T_{n_j, n_j}$, where $m+1 \leqslant j \leqslant n$, corresponds to the transient states. Since from the transient states population moves to other states, it follows quite naturally that in long run, groups of the population resides to states $T_{n_j, n_j}$, where $m+1 \leqslant j \leqslant n$, moves to states $C_{n_i}$, where $1 \leqslant i \leqslant m$ and once they are in $C_{n_i}$, they confined to those regions. Therefore, in long run, the patches corresponds to $T_{n_j, n_j}$ become without any member. This is the physical interpretation of the Theorem \ref{800}, namely, transient states correspond to the zero components of the right eigenvector $k$ (corresponding to the zero eigenvalue) which corresponds to the stable-patch structure.

\section{Conclusion}
Here we have described how to model a population when the movements of the members are restricted.  If a transition matrix $C(t)$ depends on a parameter $t$ and if we assume the the normal form (\ref{Nor}) does not change with change of values $t$, then we will also have exactly the same properties what we found here. Note that non-autonomous
age-structured inhomogeneous population models with time dependent transition matrices have been studied in \cite{I92} and \cite{RM94}. A possible interesting case could be the following: Let us consider an irreducible transition matrix $C(t)$ and let when $t = t_0$, it becomes reducible \textit{i.e.}, at the  beginning, there is a movements between all the patches but after some time, from some patches, members  moves to all other patches and they never come back. The question is, how the spectral properties changes with the limit $t \to t_0$? This and other possible generalisations will be considered for future work.

\end{document}